%% file: Modules_defined_by_Complexes5.tex
\documentclass[a4paper,10pt]{amsart}

\keywords{Gorenstein projective module; totally acyclic complex; cotorsion pair; $\lambda$-pure-injective module; deconstructible class}

\thanks{The first author is partially supported by the Spanish Government under Grants MTM2016-77445-P and MTM2017-86987-P which include FEDER funds of the EU.\\
Research of the second author supported by GA\v CR 20-13778S}

\usepackage[czech,english]{babel}
\usepackage{amssymb}
\usepackage{amsmath}
\usepackage{amsthm}
\usepackage{color}
\usepackage[shortlabels]{enumitem}
\usepackage{tikz}
\usepackage{tikz-cd}
\usepackage[all]{xy}
\input{macros}

\title[Gorenstein and $\lambda$-pure-injective modules]{THE COTORSION PAIR GENERATED BY THE GORENSTEIN PROJECTIVE MODULES AND $\lambda$-PURE-INJECTIVE MODULES}

\author{Manuel Cort\'es-Izurdiaga}
\address{Departamento de Matemática Aplicada, Universidad de Málaga, 29071, Málaga, Spain}
\email{mizurdiaga@uma.es}

\author{Jan \v{S}aroch}
\address{Charles University, Faculty of Mathematics and Physics, Department of Algebra, Sokolovsk\'a 83, 186 75 Praha 8, Czech Republic}
\email{saroch@karlin.mff.cuni.cz}

\begin{document}

\begin{abstract}
We prove that, if $\GProj$ is the class of all Gorenstein projective modules over a ring $R$, then $\mathfrak{GP}=(\GProj,\GProj^\perp)$ is a cotorsion pair. Moreover, $\mathfrak{GP}$ is complete when all projective modules are $\lambda$-pure-injective for some infinite regular cardinal $\lambda$ (in particular, if $R$ is right $\Sigma$-pure-injective); the latter condition is shown to be consistent with the axioms of ZFC modulo the existence of strongly compact cardinals.

We also thoroughly study $\lambda$-pure-injective modules for an arbitrary infinite regular cardinal $\lambda$, proving along the way that: any cosyzygy module in an injective coresolution of a $\lambda$-pure-injective module is $\lambda$-pure-injective; the cotorsion pair cogenerated by a~class of $\lambda$-pure-injective modules is cogenerated by a~set and, under an additional technical assumption, generated by a~set.

Finally, assuming the set-theoretic hypothesis that $0^\sharp$ does not exist, we prove that the category of right $R$-modules has enough $\lambda$-pure-injective objects if and only if the ring $R$ is right pure-semisimple. This, in turn, follows from a~rather surprising result that $\lambda$-pure-injectivity amounts to pure-injectivity in the absence of $0^\sharp$.
\end{abstract}

\maketitle

\section{Introduction}

Notions related with Gorenstein modules were first introduced by Auslander in~\cite{Auslander67}, where he defined the G-dimension of finitely generated modules over commutative noetherian rings, and characterized rings where all finitely generated modules have finite G-dimension in a similar way as Auslander--Buchsbaum--Serre did with commutative noetherian regular rings. These ideas were extended to non-finitely generated modules over non-commutative rings by Enochs and collaborators who defined the central notion of Gorenstein projective module and considered its natural variations, Gorenstein flat and Gorenstein injective modules (see \cite{EnochsJenda} for the general theory of Gorenstein modules). Since then, several classes of modules defined in a similar fashion as the Gorenstein ones have emerged, e.g.\ Gorenstein AC-projective modules \cite{BravoGillespieHovey}, Ding-projective modules \cite{DingLiMao} and projectively coresolved Gorenstein flat modules \cite{SarochStovicek}. All these classes are $\mathcal C$-Gorenstein projective for a~suitable class of modules $\mathcal C$, see Definition \ref{d:CGorenstein} and Proposition \ref{p:ProjectivelyCoresolvedAreCGorenstein}.

Among the problems related with these Gorenstein modules, the ones concerning the cotorsion pairs induced by them have been widely considered in the literature. One reason of this is because, due to the general theory developed by Hovey \cite{Hovey02}, these cotorsion pairs induce model structures in the category of complexes, which makes it possible to give different descriptions of the derived category. For instance, it is known that the Gorenstein modules \cite{Hovey02} or the Ding-projective modules \cite{Gillespie10} induce such model structures for right Gorenstein regular rings and Ding--Chen rings, respectively. 

The following questions naturally arise in these settings. What classes of modules appear in the cotorsion pair generated (or cogenerated) by these modules? And, is this cotorsion pair complete? Regarding the class $\GProj$ of Gorenstein projective modules, it is known that the cotorsion pair generated by them is $(\GProj,\Proj_\omega)$, where $\Proj_\omega$ consists of all modules with finite projective dimension, and that it is complete, provided the ring is right Gorenstein regular \cite[Theorem 2.26]{EnochsEstradaGarciaRozas}, i.e.\ it has finite right Gorenstein projective global dimension. Moreover, if $R$ is left coherent and right perfect, then the class $\GProj$ coincide with the class of projectively coresolved Gorenstein flat modules (see Proposition~\ref{p:ProjectivelyCoresolvedAreCGorenstein} for the definition of these modules) and, by \cite[Theorem 4.9]{SarochStovicek}, $(\GProj,\GProj^\perp)$ is a complete cotorsion pair. However, the situation over general rings is much less clear.

In this paper, we prove that, for any ring $R$, $\mathfrak{GP} = (\GProj,\GProj^\perp)$ is a cotorsion pair (see Corollary~\ref{c:CotorsionPair}). Furthermore, we show that $\mathfrak{GP}$ is complete (see Corollary \ref{c:complete} and Theorem \ref{t:complete}) provided that there exists an infinite regular cardinal $\lambda$ such that every projective module is $\lambda$-pure-injective (in particular, provided that $R$ is right $\Sigma$-pure-injective, so that we extend the aforementioned result about left coherent and right perfect rings, since these rings are right $\Sigma$-pure-injective by \cite[Proposition 11.1]{JensenLenzing}). Finally, assuming a set-theoretical hypothesis concerning compact cardinals (see Section~\ref{sec:5} for details), we show that the cotorsion pair $\mathfrak{GP}$ is complete for any ring $R$.

We use $\lambda$-pure-injective modules, where $\lambda$ is an infinite regular cardinal, to study when the cotorsion pair generated by the Gorenstein projective modules is complete. The $\lambda$-pure-injective modules are the natural extensions of the classical ($\aleph_0$-)pure-injective modules that appear when one considers linear systems of equations with less than $\lambda$ equations, instead of finite systems of equations, see \cite{Saroch15}. These systems characterize the $\lambda$-purity, which is a~natural notion in the setting of $\lambda$-presentable and, more generally, $\lambda$-accessible categories, cf.~\cite{AdamekRosicky}. We obtain, in Section~\ref{sec:4}, some results about $\lambda$-pure-injective modules and about the cotorsion pairs cogenerated by families of them. For instance, that any cosyzygy module in an injective coresolution of a~$\lambda$-pure-injective module is $\lambda$-pure-injective (Proposition~\ref{p:cosyz}); or that, if $\mathcal C$ is a~class of $\lambda$-pure-injectives, then ${^\perp}\mathcal C$ is closed under $\lambda$-pure epimorphic images and $\lambda$-pure submodules (Theorem~\ref{t:closureprop}); finally, under certain circumstances, the class ${^\perp}\mathcal C$ is shown to be even deconstructible (Corollary~\ref{c:CotorsionPairCogeneratedLambdaPureInjectives}).

We finish the paper studying the problem of the existence of enough $\lambda$-pure-injective modules. It is a~classical result that every module purely embeds into a~pure-injective module, see~\cite[Theorem 4.3.18]{Prest}. So it is natural to ask if this remains true for $\lambda$-pure-injective modules and $\lambda$-pure embeddings where $\lambda$ is uncountable.  Assuming the set-theoretical hypothesis ``$0^{\sharp}$ does not exist'', we show in Theorem~\ref{t:0sharp} that there is such an uncountable (regular) $\lambda$ if and only if the ring is right pure-semisimple. The proof of this fact is based on a~slightly surprising Proposition~\ref{p:0sharp} which asserts that $\lambda$-pure-injectivity coincides with pure-injectivity in the absence of $0^\sharp$. Its proof, in turn, relies on a~general result from the recent paper \cite{BS23} and a~somewhat technical Lemma~\ref{l:wdiamond} which utilizes the \emph{Weak Diamond Principle}.

\section{Preliminaries}
\label{sec:preliminaries}

Given a set $A$, we denote by $|A|$ its cardinality. If $\mu$ and $\lambda$ are cardinals, $\mu^{<\lambda}$ is the cardinality of the set of all maps $f\colon \alpha \rightarrow \mu$ with $\alpha < \lambda$. An infinite cardinal $\kappa$ is called \emph{singular} if there exists a cardinal $\mu < \kappa$ and a family of cardinals $\{\kappa_\alpha \mid \alpha < \mu\}$ with $\kappa_\alpha < \kappa$ for each $\alpha$, such that $\kappa = \sum_{\alpha < \mu}\kappa_\alpha$. A \emph{regular} cardinal is a non-singular cardinal.

Throughout this paper, $R$ denotes a (not necessarily commutative) ring with unit. A~module means a~right $R$-module. We use $\Proj$, $\Flat$ and $\Inj$ to denote the classes of projective, flat and injective modules respectively.

Let $\mathcal A$ be a class of modules. The class $\mathcal A$ of modules is called \emph{resolving} (resp.\ \emph{coresolving}) if it contains all projective (resp.\ injective) modules and it is closed under extensions and kernels of epimorphisms (resp.\ cokernels of monomorphisms). We denote by $\mathcal A^\perp$ (resp.\ ${}^\perp\mathcal{A}$) the corresponding Ext-orthogonal classes of $\mathcal{A}$, that is,
\begin{displaymath}
\mathcal A^{\perp} = \{M \in \Modr R \mid \Ext_R^1(A,M)=0 \quad \forall A \in \mathcal A\}
\end{displaymath}
and
\begin{displaymath}
{}^\perp\mathcal{A} = \{M \in \Modr R \mid \Ext_R^1(M,A)=0 \quad \forall A \in \mathcal A\}.
\end{displaymath}
An \emph{$\mathcal A$-precover} (resp.\ an \emph{$\mathcal A$-preenvelope}) of a module $M$ is a morphism $\varphi\colon A \rightarrow M$ (resp.\ $\psi\colon M \rightarrow A$) with $A \in \mathcal A$ and such that $\Hom_R(A^\prime,\varphi)$ (resp.\ $\Hom_R(\psi,A^\prime)$) is epic for each $A^\prime \in \mathcal A$. The precover (resp.\ preenvelope) is \emph{special} if it is epic (resp.\ monic) and $\Ker \varphi \in \mathcal A^\perp$ (resp.\ $\Coker \psi \in {^\perp}\mathcal A$). A \emph{cotorsion pair} in $\Modr R$ is a pair of classes of modules, $(\mathcal F,\mathcal C)$, such that $\mathcal C=\mathcal F^\perp$ and $\mathcal F = {^\perp}\mathcal C$. The cotorsion pair is called \emph{complete} if every module has a~special $\mathcal F$-precover, equivalently a~special $\mathcal C$-preenvelope, and \emph{hereditary} if $\mathcal F$ is resolving and $\mathcal C$ is coresolving.

A useful notion related with complete cotorsion pairs is that of deconstructible class. An \emph{$\mathcal A$-filtration} of a~module $M$ is a~continuous chain of submodules of $M$, $(M_\alpha\mid \alpha\leq \sigma)$, indexed by some ordinal $\sigma$, such that $M_0 = 0$, $\frac{M_{\alpha+1}}{M_\alpha} \in \mathcal A$ for each $\alpha < \sigma$ and $M = M_\sigma$. We denote by $\Filt(\mathcal A)$ the class of all $\mathcal A$-filtered modules, i.e.\ modules possessing an $\mathcal A$-filtration. We say that $\mathcal A$ is \emph{closed under filtrations} if $\Filt(\mathcal A) \subseteq \mathcal A$, and \emph{deconstructible} if there exists a set of modules $\mathcal S$ such that $\mathcal A = \Filt(\mathcal S)$. Since $\Filt(\mathcal A)$ is closed under filtrations, $\mathcal A$ is deconstructible if and only if $\mathcal A$ is closed under filtrations and there exists a set $\mathcal S \subseteq \mathcal A$ such that $\mathcal A \subseteq \Filt(\mathcal S)$. Moreover, from Eklof lemma, \cite[Lemma 6.2]{GobelTrlifaj21} and \cite[Theorem 6.11]{GobelTrlifaj21}, it follows that a~cotorsion pair $(\mathcal F,\mathcal C)$ is complete provided that $\mathcal F$ is deconstructible.

Given a complex $X$ of modules,
\begin{displaymath}
\begin{tikzcd}
\arrow{r} \cdots & X^{n-1} \arrow{r}{d_X^{n-1}} & X^n \arrow{r}{d_X^n} & X^{n+1} \arrow{r} & \cdots,
\end{tikzcd}
\end{displaymath}
we denote by $\Z^n(X)$ the kernel of $d_X^n$. A module $M$ is \textit{Gorenstein projective} if there exists an exact complex $X$ consisting of projective modules, such that $\Z^n(X) \in {^\perp}\Proj$ for each $n \in \mathbb Z$ and $M \cong \Z^0(X)$. We denote by $\GProj$ the class of all Gorenstein projective modules. We are interested in the following extension of the notion of Gorenstein projective module.

\begin{definition}\label{d:CGorenstein}
Let $\mathcal C$ be a class of modules. A module $M$ is called:
\begin{itemize}
\item \emph{$\mathcal C$-Gorenstein projective} if there exists an exact complex $X$ consisting of projective modules such that $\Z^n(X) \in {^\perp}\mathcal C$ for each $n \in \mathbb Z$ and $M \cong \Z^0(X)$. We denote by $\mathcal \GProj_\mathcal{C}$ the class of all such modules.

\item \emph{Strongly $\mathcal C$-Gorenstein projective} if $M \in {^\perp}{\mathcal C}$ and there exists a short exact sequence
\begin{displaymath}
\begin{tikzcd}
0 \arrow{r} & M \arrow{r} & P \arrow{r} & M \arrow{r} & 0
\end{tikzcd}
\end{displaymath}
with $P$ projective.
\end{itemize}
\end{definition}

It is clear that each strongly $\mathcal C$-Gorenstein projective module is $\mathcal C$-Gorenstein projective. On the other hand, it is easy to see (cf.\ \cite[Lemma~4.1]{SarochStovicek}) that the direct sum of all the cycles in a complex $X$ of projective modules satisfying $\Z^n(X) \in {^\perp}\mathcal C$ is a strongly $\mathcal C$-Gorenstein projective module. Hence, each $\mathcal C$-Gorenstein projective is a direct summand of a~strongly $\mathcal C$-Gorenstein projective module. Moreover, if $\Proj\subseteq\mathcal C$, then $\GProj_{\mathcal C}$ is precisely the class of all direct summands of strongly $\mathcal C$-Gorenstein projective modules (since we know, e.g.\ by Theorem~\ref{t:CotPair}, that the class of $\mathcal C$-Gorenstein projective modules is closed under direct summands).

Clearly, the Gorenstein projective modules are the $\Proj$-Gorenstein modules and the Ding-projective modules \cite[Definition 3.7]{Gillespie10} are the $\Flat$-Gorenstein modules. Other variations of the concept of Gorenstein projective module are $\mathcal C$-Gorenstein projective for a suitable class of modules $\mathcal C$ as well:

\begin{proposition}\label{p:ProjectivelyCoresolvedAreCGorenstein}
\begin{enumerate}
\item The projectively coresolved Gorenstein flat modules \cite[p. 15]{SarochStovicek} are the $\langle \Flat \rangle$-Gorenstein projective modules, where $\langle \Flat \rangle$ denotes the smallest definable subcategory of $\Modr R$ (i.e., closed under pure submodules, products and direct limits) that contains $\Flat$.

\item The Gorenstein AC-projective modules \cite[p. 30]{BravoGillespieHovey} are the $\mathcal L$-Gorenstein projective modules, where $\mathcal L$ is the class of all level modules (a~module $L$ is \emph {level} if $\Tor_1^R(L,F) = 0$ for each left module $F$ with a~projective resolution consisting of finitely generated projectives).
\end{enumerate}
\end{proposition}

\begin{proof}
(1) Follows from \cite[Corollary 4.5]{SarochStovicek}.

(2) Follows from \cite[Corollary A.7]{BravoGillespieHovey}.
\end{proof}

\section{Cotorsion pair generated by $\mathcal C$-Gorenstein projective modules}

In this section we prove that the cotorsion pair generated by the class of $\mathcal C$-Gorenstein projective modules is $(\mathcal \GProj_{\mathcal C}, \GProj_{\mathcal C}^\perp)$. In particular, we get that the cotorsion pair generated by the Gorenstein projective modules is $(\GProj,\GProj^\perp)$. The proof is based on the following result:

\begin{lemma} \label{l:CotorsionPairGenerated}
Let $\mathcal B$ be a class of $R$-modules closed under kernels of epimorphisms and products. Suppose that there exists a generator $G$ such that $G^{(\kappa)} \in \mathcal B$ for each cardinal $\kappa$. Then, for each infinite cardinal $\kappa$, every $\kappa$-generated module $M \in {^\perp}\mathcal B$ has a~$\mathcal B$-preenvelope $\varphi\colon M \rightarrow G^{(\kappa)}$.

If, in addition, $\mathcal B$ contains a cogenerator of $\ModR$ and $G \in {^\perp}\mathcal B$, then $\varphi$ is a~special $\mathcal B$-preenvelope.
\end{lemma}

\begin{proof}
Suppose that the result is not true for some $\kappa$-generated module $M$. Then, for each $\psi \in \Hom_R(M,G^{(\kappa)})$, there exist $B_\psi \in \mathcal B$ and some $f_\psi\in \Hom_R(M, B_\psi)\setminus\Img(\Hom_R(\psi,B_\psi))$. Set \[B=\prod_{\psi \in \Hom_R(M,G^{(\kappa)})}B_\psi\] and let $f\colon M \rightarrow B$ be induced by the family $(f_\psi\mid \psi \in \Hom_R(M,G^{(\kappa)}))$. Take an epimorphism $\pi\colon G^{(\mu)} \rightarrow B$ for some infinite cardinal $\mu$, which can be assumed to be greater than or equal to $\kappa$. Since $G^{(\mu)}, B \in \mathcal B$ and $\mathcal B$ is closed under kernels of epimorphisms by our hypotheses, $\Ker \pi \in \mathcal B$ as well. Since $M \in \mathcal {^\perp}\mathcal B$, there exists a~homomorphism $\rho \colon M \to G^{(\mu)}$ such that $\pi\rho = f$. Since $M$ is $\kappa$-generated, $\rho$ factors through $G^{(\kappa)}$, that is, there exists a~$\bar \rho \in \Hom_R(M, G^{(\kappa)})$ such that $\iota \bar \rho = \rho$, where $\iota$ is the canonical inclusion of $G^{(\kappa)}$ into $G^{(\mu)}$. 

In particular, if we denote by $\pi_{\bar \rho}\colon B \to B_{\bar\rho}$ the $\bar\rho$th canonical projection, then $f_{\bar\rho} = \pi_{\bar\rho}f = \pi_{\bar\rho} \pi \iota\bar\rho \in \Img(\Hom_R(\bar\rho,B_{\bar\rho}))$, a~contradiction. We have shown the existence of a~$\mathcal B$-preenvelope $\varphi\colon M \rightarrow G^{(\kappa)}$.

In order to prove the last assertion, notice that $\varphi$ has to be monic since $\mathcal B$ contains a~cogenerator, $\mathcal B$ is closed under products and $\varphi$ is a~$\mathcal B$-preenvelope. Moreover, since $G^{(\kappa)} \in \mathcal {^\perp}\mathcal B$ and $\varphi$ is one-one, $\Coker \varphi \in \mathcal {^\perp}\mathcal B$ as well. That is, $\varphi$ is a special $\mathcal B$-preenvelope.
\end{proof}

We want to apply Lemma~\ref{l:CotorsionPairGenerated} to $\GProj_{\mathcal C}$. We will use the following: 

\begin{lemma}\label{l:PropertiesCGorenstein}
Let $\mathcal C$ be a class of modules containing all projective modules.
\begin{enumerate}
\item If $M \in \mathcal \GProj_{\mathcal C}^{\perp}$, then $\Ext_R^n(H,M)=0$ for each positive integer $n$ and $H\in\mathcal \GProj_{\mathcal C}$.

\item $\GProj_{\mathcal C}^\perp$ contains $\mathcal C$ and is resolving and coresolving.
\end{enumerate}
\end{lemma}

\begin{proof}
(1) Given $H \in \GProj_{\mathcal C}$, there exists a projective resolution $X$ of $H$ with $\Z^n(X) \in \GProj_{\mathcal C}$ for each $n \leq -1$. Consequently, if $M \in \GProj_{\mathcal C}^{\perp}$, then $ 0 =\Ext_R^1(\Z^{-n+1}(X),M) \cong \Ext_R^n(H,M)$ for each $n > 1$.

(2) $\GProj_{\mathcal C}^\perp$ is coresolving by (1). In order to prove that it is resolving, take a~short exact sequence
\begin{equation}\label{e:Ses1}
\begin{tikzcd}
0 \arrow{r} & A \arrow{r} & B \arrow{r} & C \arrow{r} & 0
\end{tikzcd}
\end{equation}
with $B$ and $C$ belonging to $\GProj_{\mathcal C}^\perp$. Fix $H \in \GProj_{\mathcal C}$ and take a~short exact sequence
\begin{equation}\label{e:Ses2}
\begin{tikzcd}
0 \arrow{r} & H \arrow{r} & P \arrow{r} & G \arrow{r} & 0
\end{tikzcd}
\end{equation}
with $P \in \Proj$ and $G \in \GProj_{\mathcal C}$. Applying $\Hom_R(G,-)$ to (\ref{e:Ses1}) we get the exact sequence
\begin{displaymath}
\begin{tikzcd}
\Ext_R^1(G,C) \arrow{r} & \Ext_R^2(G,A) \arrow{r} & \Ext_R^2(G,B)
\end{tikzcd}
\end{displaymath}
from which $\Ext_R^2(G,A)=0$. But (\ref{e:Ses2}) implies that $\Ext_R^1(H,A) \cong \Ext_R^2(G,A) =0$.

Finally, the inclusion $\GProj_{\mathcal C} \subseteq {^\perp}\mathcal C$ implies that $\mathcal C \subseteq \GProj_{\mathcal C}^\perp$.
\end{proof}

Now, the main theorem of this section:

\begin{theorem}\label{t:CotPair}
Let $\mathcal C$ be a class of modules containing all projective modules. Then $(\GProj_{\mathcal C},\GProj_{\mathcal C}^{\perp})$ is a hereditary cotorsion pair with $\GProj_{\mathcal C}^\perp$ a resolving class.
\end{theorem}

\begin{proof}
The class $\GProj_{\mathcal C}^\perp$ is resolving and coresolving by Lemma \ref{l:PropertiesCGorenstein}. Then we just have to prove that ${^\perp}(\GProj_{\mathcal C}^\perp)$ is contained in $\GProj_{\mathcal C}$. Take $M \in {^\perp}{(\GProj_{\mathcal C} ^\perp)}$ and let us construct a complex $X$ of projectives with $\Z^n(X) \in {}^{\perp}\mathcal C$ for each $n < \omega$ and $\Z^0(X) = M$.

Since $\GProj_{\mathcal C}^\perp$ is resolving, closed under products and contains all projective modules by Lemma \ref{l:PropertiesCGorenstein}, we can apply Lemma \ref{l:CotorsionPairGenerated} with $\mathcal B =  \GProj_{\mathcal C}^\perp$ to get a~special $\GProj_{\mathcal C}^\perp$-preenvelope $M \rightarrow X^0$ with $X^0$ free of rank $\kappa$, for $\kappa$ the number of generators of $M$. We can repeat this argument recursively to construct $X^n$ for each $n > 0$. These satisfy that $X^{n} \rightarrow X^{n+1} \rightarrow X^{n+2}$ is exact, $X^n$ is free, $\Z^n(X) \in {^\perp}{(\GProj_{\mathcal C}^\perp)}$ for each $n \geq 0$, and $\Z^0(X) = M$.

Now take a projective presentation $X^{-1} \rightarrow M$, where $X^{-1}$ can be taken to be free. Since $\GProj_{\mathcal C}^\perp$ is coresolving, ${^\perp}(\GProj_{\mathcal C}^\perp)$ is resolving, so that the kernel of this projective presentation belongs to ${^\perp}(\GProj_{\mathcal C}^\perp)$. We can repeat this argument recursively to construct $X^n$ for each $n<-1$. 
\end{proof}

As a consequence of this result we obtain that $(\GProj,\GProj^{\perp})$ is a hereditary cotorsion pair. Moreover, we can compute the left hand class of the cotorsion pairs generated by the classes considered in Proposition \ref{p:ProjectivelyCoresolvedAreCGorenstein}:

\begin{corollary}\label{c:CotorsionPair}
The following pair of classes are hereditary cotorsion pairs with the right hand classes being resolving:
\begin{enumerate}
\item $(\GProj,\GProj^\perp)$.

\item $(\PGFlat,\PGFlat^\perp)$, where $\PGFlat$ is the class of all projectively coresolving Gorenstein flat modules.

\item $(\AGProj,\AGProj^\perp)$, where $\AGProj$ is the class of all Gorenstein $AC$-projective modules.

\item $(\DGProj,\DGProj^\perp)$, where $\DGProj$ is the class of all Din-projective modules.
\end{enumerate}

\end{corollary}

\begin{proof}
(1) and (4) follows immediately from the previous theorem. (2) and (3) follows from the previous theorem as well using the description of Gorenstein AC-projective and Ding-projective modules given in Proposition \ref{p:ProjectivelyCoresolvedAreCGorenstein}.
\end{proof}

\section{Cotorsion pairs cogenerated by $\lambda$-pure-injective modules}
\label{sec:4}

Let $\lambda$ be an infinite regular cardinal. In this section we prove some results about $\lambda$-pure-injective modules which will be used in the next one in order to prove that the cotorsion pair generated by the Gorenstein projective modules is complete when every projective module is $\lambda$-pure-injective.

Recall that an embedding $M\subseteq N$ of $R$-modules is called \emph{$\lambda$-pure} if each system consisting of $R$-linear equations with parameters in $M$ and having cardinality $<\lambda$ has a solution in $M$ provided that it has a solution in $N$. A short exact sequence $0 \rightarrow K \xrightarrow{f} M \xrightarrow{g} L \rightarrow 0$ is \emph{$\lambda$-pure} if the embedding $f(K) \subseteq M$ is $\lambda$-pure. In this case, $g$ is called a \emph{$\lambda$-pure epimorphism}. Notice that $\lambda$-pure embeddings, and correlatively $\lambda$-pure short exact sequences, are rather frequent. For instance, for the direct limit $L$ of any $\lambda$-directed system $\mathcal S$ of modules, the canonical pure short exact sequence $0\to K \overset{\subseteq}{\to} \bigoplus_{M\in \mathcal S} M \to L \to 0$ is actually $\lambda$-pure. Of course, if $\lambda = \aleph_0$, we obtain the usual notion of purity. As usual, we will denote by $\textrm{PExt}_R^1(N,M)$ the derived functors of the Hom functors with respect to the pure-exact structure of $\Modr R$. Notice that $\textrm{PExt}_R^1(N,M)=0$ if and only if every pure-exact sequence $0 \rightarrow M \rightarrow L \rightarrow N \rightarrow 0$ splits.

The following result is an extension to $\lambda$-purity of the well known purification lemma (see \cite[Lemma 5.3.12]{EnochsJenda}).

\begin{lemma}\label{l:purification}
Let $\lambda$ be an infinite regular cardinal and $\kappa$ an infinite cardinal such that $|R|\leq \kappa$ and $\kappa^{<\lambda}=\kappa$. Then, for any $\kappa$-generated submodule $K$ of a module $M$, there exists a~$\kappa$-generated $\lambda$-pure submodule $\overline K$ of $M$ containing $K$.
\end{lemma}

\begin{proof}
We construct, using transfinite recursion, a chain of $\kappa$-generated submodules $\{A_\alpha \mid \alpha < \lambda\}$ of $M$ such that $K \leq A_0$ and any system with less than $\lambda$ equations and parameters in $\bigcup_{\gamma < \alpha}A_\gamma$ that has a solution in $M$, has solution in $A_\alpha$ for any $\alpha < \lambda$ non-zero.

Set $A_0=K$ and suppose that for some $\beta < \lambda$ we have constructed $A_\alpha$ for each $\alpha < \beta$. Set $A^\prime_\beta=\bigcup_{\alpha < \beta}A_\alpha$. Then $|A^\prime_\beta|\leq \kappa$, so that, because of the hypotheses on $\kappa$, the set $\mathcal S$ of systems of $R$-linear equations with less than $\lambda$ equations and parameters in $A^\prime_\beta$ having a solution in $M$, has cardinality less than or equal to $\kappa$. So, the submodule $A_\beta$ of $M$ generated by $A^\prime_\beta$ and the set consisting of one solution for every system belonging to $\mathcal S$ has cardinality less than or equal to $\kappa$ as well. 

Clearly, the submodule $\overline K=\bigcup_{\alpha < \lambda}A_\alpha$ satisfies the desired properties.
\end{proof}

We say that a~filtration $(M_\alpha\mid \alpha\leq\sigma)$ of $M$ is \emph{pure} if $M_\alpha$ is pure in $M$ for every $\alpha\leq\sigma$. Using the lemma above, we can build a~pure filtration of $M$ with \emph{almost} all members $\lambda$-pure in $M$. This will be useful later on in Section~\ref{s:lambda}.

\begin{lemma}\label{l:purefilt} Let $\lambda$ be an infinite regular cardinal and $\kappa$ an infinite cardinal such that $|R|\leq \kappa$ and $\kappa^{<\lambda}=\kappa$. Let $M$ be a~$\kappa^+$-generated module. Then there exists a~pure filtration $(M_\alpha\mid \alpha\leq\kappa^+)$ of $M$ such that $|M_\alpha|\leq\kappa$ for each $\alpha<\kappa^+$; moreover, $M_\alpha$ is $\lambda$-pure in $M$ whenever $\alpha\leq\kappa^+$ is either non-limit or $\cf(\alpha)\geq\lambda$.
\end{lemma}

\begin{proof} We easily build the filtration by repeatedly using Lemma~\ref{l:purification} in non-limit steps and forming unions otherwise. For $\alpha<\kappa^+$ limit, the $\lambda$-purity of thus constructed $M_\alpha$ in $M$ follows immediately if $\cf(\alpha)\geq \lambda$; if $\cf(\alpha)<\lambda$, then $M_\alpha$ need not be $\lambda$-pure in $M$ but it is going to be pure nonetheless, hence the constructed filtration is, indeed, pure.
\end{proof}

A module $F$ is called \emph{$\lambda$-pure-injective} if it is injective relative to all $\lambda$-pure embeddings. This is equivalent to saying that each system consisting of $R$-linear equations with parameters in $F$ has a solution in $F$ provided that each of its subsystems of cardinality $<\lambda$ has one. Only a little is known about $\lambda$-pure-injective modules for uncountable $\lambda$. Recently, a few results in this direction have appeared in \cite{SarochTrlifaj19}. In what follows, we will use that, if $\lambda$ is a sufficiently large cardinal, many modules suddenly become $\lambda$-pure-injective. However, if $\lambda = \aleph_1$, we get nothing new as the next result (generalizing \cite{Megibben74} and the first footnote in \cite[pg.\ 221]{SarochTrlifaj19}) shows.

\begin{proposition} \label{p:aleph_1} Let $R$ be a ring. Each $\aleph_1$-pure-injective module is pure-injective.\footnote{It is plausible, that the conclusion remains valid if $\aleph_1$ is replaced by $\aleph_n$ for an arbitrary $n<\omega$.}
\end{proposition}

\begin{proof} Let $M$ be an $\aleph_1$-pure-injective module. We have to show that each pure short exact sequence $0\to M\to B\to N\to 0$ splits; otherwise put, that $\PExt_R^1(N, M) = 0$ for any module $N$. Since every module $N$ is the direct limit of an $\aleph_1$-directed system consisting of countably presented modules, in particular it is an $\aleph_1$-pure-epimorphic image of a direct sum consisting of countably presented modules, we just need to check that $\PExt_R^1(C,M) = 0$ for each countably presented module $C$. The rest follows from the $\aleph_1$-pure-injectivity of $M$.

Notice that $\PExt_R^1(N,M) = 0$ holds for any pure-projective module $N$, and hence also for any $N$ which is an $\aleph_1$-pure-epimorphic image of a pure-projective
module. In particular, this is the case of any $N$ which is Mittag-Leffler, cf.\ \cite[Theorem~3.14]{GobelTrlifaj21}. Let $0\to M\to B\overset{f}{\to} C\to 0$ be a pure short exact sequence with $C$ countably presented. Using \cite[Theorem~4.2(1)]{Saroch17} for $\theta = \aleph_0$, we obtain a Mittag-Leffler module $L$ such that $\Hom_R(L,f)$ is onto if and only if $f$ splits. However, we already know that $\PExt_R^1(L,M) = 0$, whence $\Hom_R(L,f)$ is surjective. This finishes our proof.
\end{proof}

Now we give a generalization of \cite[Lemma~6.20]{GobelTrlifaj21} for which we need an auxiliary lemma. Recall that, given a module $M$, a nonempty set~$I$ and a filter $\mathcal F$ on $I$, we denote by $M^I/\mathcal F$ the \emph{reduced power of $M$ modulo $\mathcal F$}, i.e.\ $M^I/\mathcal F = M^I/Z$ where $Z = \{(m_i)_{i\in I}\in M^I \mid \{i\in I\mid m_i = 0\}\in\mathcal F\}$. For instance, if $\mathcal F$ consists of cofinite subsets of $I$, then $Z = M^{(I)}$. In any case, there is the \emph{diagonal embedding} $\Delta\colon m\mapsto (m)_{i\in I}/\mathcal F$ of $M$ into the reduced power. This embedding is $\lambda$-pure whenever $\mathcal F$ is $\lambda$-complete, i.e.\ closed under the intersection of $<\lambda$ sets. Since each filter is $\aleph_0$-complete by definition, $\Delta$ is always a pure embedding.

\begin{lemma}\label{l:solext} Let $M$ be an $R$-module, $\lambda$ a regular infinite cardinal and $\Omega$ a set consisting of $R$-linear equations with parameters from $M$. Let us denote by $X$ the set of variables occurring in the equations from $\Omega$.

Assume that each subset $S$ of $\Omega$ such that $|S|<\lambda$ has a solution in $M$. Then there exists a $\lambda$-complete filter on the set $I = M^X$ such that $\Omega$ has a solution in the reduced power $P = M^I/\mathcal F$ into which $M$ is embedded using the diagonal embedding.
\end{lemma}

\begin{proof} We define $\mathcal F$ as the filter on $I$ generated by the sets $E_S = \{e\in I \mid (e(x))_{x\in X}\mbox{ is a solution of }S\}$ where $S$ runs through subsets of $\Omega$ of cardinality $<\lambda$. Since each $E_S$ is nonempty by our assumption and $\lambda$ is infinite regular, we see that $\mathcal F$ is $\lambda$-complete. We claim that $x\mapsto (e(x))_{e\in I}/\mathcal F$ where $x$ runs through $X$ defines a~solution of $\Omega$ in $P$. By the property of reduced products (cf.\ \cite[Proposition~3.3.3]{Prest}), this amounts to showing that, for each $\eta\in\Omega$, \[\{e\in I \mid (e(x))_{x\in X}\mbox{ is a solution of }\eta\}\in\mathcal F.\] This trivially follows from the definition of~$\mathcal F$
\end{proof}

\begin{proposition}\label{p:cosyz} Let $\lambda$ be an infinite regular cardinal and $P$ a $\lambda$-pure-injective $R$-module. Then any cosyzygy module in an injective coresolution of $P$ is $\lambda$-pure-injective as well.
\end{proposition}

\begin{proof} It clearly suffices to prove that, if $P\subseteq E$ is an injective envelope of $P$, then $M = E/P$ is $\lambda$-pure-injective. Let $\Omega$ be a set of $R$-linear equations such that each of its subsets of cardinality $<\lambda$ has a solution in $M$. Using Lemma~\ref{l:solext}, we obtain a $\lambda$-complete filter on a set $I$ such that $\Omega$ has a solution in $M^I/\mathcal F$ into which $M$ is embedded using the diagonal embedding $\Delta_M$. Using the functoriality of the diagonal embedding (having $I$ and $\mathcal F$ fixed), see \cite[pg.\ 103]{Prest}, we get the following commutative diagram with exact rows

\[\xymatrix{0 \ar[r] & P \ar[r] \ar[d]^{\Delta_P} & E \ar[r] \ar[d]^{\Delta_E} & M \ar[r] \ar[d]^{\Delta_M} & 0 \\
0 \ar[r] & P^I/\mathcal F \ar[r] & E^I/\mathcal F \ar[r] & M^I/\mathcal F \ar[r] & 0.\!\!
}\]

The monomorphism $\Delta_P$ is split since it is $\lambda$-pure (because $\mathcal F$ is $\lambda$-complete) and $P$ is $\lambda$-pure-injective. We follow the proof of \cite[Lemma~6.20]{GobelTrlifaj21} with $\eta_P$ replaced by $\Delta$ and $F$ replaced by $M$ to deduce that $\Delta_M$ is a split monomorphism as well. Let us denote by $\psi\colon M^I/\mathcal F\to M$ the respective retraction. Then the $\psi$-image of the solution of $\Omega$ in $M^I/\mathcal F$ is a solution of $\Omega$ in $M$. Since $\Omega$ was arbitrary, we have proved that $M$ is $\lambda$-pure-injective.
\end{proof}

The above proposition allows us to prove the following

\begin{theorem}\label{t:closureprop} Let $\lambda$ be an infinite regular cardinal and let $\mathcal C$ be a class of $\lambda$-pure-injective modules. Put $\mathcal A = {}^\perp\mathcal C$. Then:
\begin{enumerate}
	\item $\mathcal A$ is closed under $\lambda$-pure-epimorphic images and $\lambda$-pure submodules;
	\item there exists a $\lambda$-pure-injective module $C$ such that $\mathcal A = {}^\perp C$.
\end{enumerate}
\end{theorem}

\begin{proof} $(1)$. The closure under $\lambda$-pure-epimorphic images follows immediately from the fact that $\mathcal C$ consists of $\lambda$-pure-injective modules. On the other hand, let $B$ be a $\lambda$-pure submodule of $A\in\mathcal A$ and let $P\in\mathcal C$ be arbitrary. Consider a short exact sequence \[0\to P\to E \overset{\pi}{\to} M\to 0\] with $E$ injective. We know that $M$ is $\lambda$-pure-injective by Proposition~\ref{p:cosyz}. It follows that each homomorphism $f\in\Hom_R(B, M)$ can be extended to $g\in\Hom_R(A, M)$. Since $A\in\mathcal A$ and $P\in\mathcal C$, we can factorize $g$ through $\pi$ to obtain a homomorphism whose restriction to $B$ is a factorization of $f$ through $\pi$. This shows that $B\in\mathcal A$.

$(2)$. Recall that, by \cite[Theorem~2.34]{AdamekRosicky}, there exists a regular cardinal $\kappa\geq\lambda$ such that each module is the directed union of a $\kappa$-directed system consisting of $<\kappa$-presented $\lambda$-pure submodules. Let $\mathcal D$ be a~subset of $\mathcal C$ such that ${}^\perp \mathcal D$ contains exactly the same $<\kappa$-presented modules as $\mathcal A = {}^\perp\mathcal C$ does; such a~set $\mathcal D$ exists since the full subcategory of $\ModR$ consisting of $<\kappa$-presented modules is skeletally small. Put $C = \prod_{D\in\mathcal D} D$. Then clearly $\mathcal A \subseteq {}^\perp C$.

To show the other inclusion, let $B\in {}^\perp C$ be arbitrary. Then $B$ is the directed union of a $\kappa$-directed system $\mathcal S$ consisting of $<\kappa$-presented $\lambda$-pure submodules. By $(1)$, we know that ${}^\perp C$ is closed under $\lambda$-pure submodules, whence $\mathcal S\subseteq {}^\perp C$ and consequently $\mathcal S\subseteq\mathcal A$ since $\mathcal S$ consists of $<\kappa$-presented modules. It follows by $(1)$ that $B\in\mathcal A$ since the canonical map $\bigoplus_{M\in\mathcal S} M\to B$ is a $\lambda$-pure-epimorphism (it is even $\kappa$-pure since $\mathcal S$ is $\kappa$-directed).
\end{proof}

Notice that it follows from the proof that each class of the form ${}^\perp\mathcal C$, where $\mathcal C$ consists of $\lambda$-pure-injective modules, is uniquely determined by its $<\kappa$-presented members (for the cardinal $\kappa$ from the proof). In fact, if we denote by $\mathcal S$ the representative set of $<\kappa$-presented modules in ${}^\perp\mathcal C$, then ${}^\perp\mathcal C$ consists precisely of direct limits of $\kappa$-directed systems comprising of members from $\mathcal S$ (and monomorphisms between them). In particular, as in the special case when $\lambda = \aleph_0$, cotorsion pairs which are cogenerated by a~class of $\lambda$-pure-injective modules form a~set.

We finish this section with two applications of this result. Another one will be given at the end of Section 6. First, it is known that every cotorsion pair cogenerated by a class of pure-injective modules is generated by a set (this follows from \cite[Theorem 8]{EklofTrlifaj}). For cotorsion pairs cogenerated by $\lambda$-pure-injective modules we get the following slightly weaker result:

\begin{corollary}\label{c:CotorsionPairCogeneratedLambdaPureInjectives}
Let $\lambda$ be an infinite regular cardinal and $(\mathcal F,\mathcal C)$ be a~cotorsion pair cogenerated by a~class of $\lambda$-pure-injective modules. Assume, moreover, that
\begin{enumerate}
\item[$(\dagger)$] $\mathcal F$ is closed under direct limits\footnote{This condition is satisfied, e.g., if $\mathcal F$ is closed under taking cokernels of pure monomorphisms, cf.\ the proof of \cite[Proposition~3.1]{G17}.}.
\end{enumerate}
Then $\mathcal F$ is deconstructible. In particular, $(\mathcal F,\mathcal C)$ is complete.
\end{corollary}

\begin{proof}
Let $\kappa$ be an infinite cardinal satisfying $\kappa \geq |R|$ and $\kappa^{< \lambda}=\kappa$. Let $F \in \mathcal F$. We prove that $F$ is filtered by $\leq \kappa$-generated modules in $\mathcal F$. Suppose that $F=\{x_\gamma \mid \gamma < \mu\}$ for some cardinal $\mu$ greater than $\kappa$. We construct the filtration $(F_\alpha \mid \alpha \leq \mu)$ of $F$ satisfying $x_\alpha \in F_{\alpha+1}$, $\frac{F_{\alpha+1}}{F_\alpha}, \frac{F}{F_\alpha} \in \mathcal F$ and $\left|\frac{F_\alpha+1}{F_\alpha}\right| \leq \kappa$ for each $\alpha < \mu$.

Set $F_0=0$. If $\alpha$ is successor, say $\alpha = \beta+1$, apply Lemma \ref{l:purification} to get a $\lambda$-pure submodule $\frac{F_\alpha}{F_\beta}$ of $\frac{F}{F_\beta}$ containing $x_\beta+F_\beta$ and with cardinality less than or equal to~$\kappa$. By the preceding theorem, $\frac{F_\alpha}{F_\beta}, \frac{F}{F_\alpha} \in \mathcal F$. 

If $\alpha$ is limit, set $F_\alpha = \bigcup_{\gamma < \alpha}F_\gamma$. Then $F_\alpha \in \mathcal F$ by Eklof Lemma and $\frac{F}{F_\alpha} \in \mathcal F$ by~$(\dagger)$.
\end{proof}

Let us note that the extra assumption $(\dagger)$ in the corollary above is automatically satisfied in any model of ZFC without $0^\sharp$ by Proposition~\ref{p:0sharp}. Furthermore, if we weaken the assumption on the cotorsion pair $(\mathcal F, \mathcal C)$ to ``$\mathcal F$ is closed under $\lambda$-pure submodules and $\lambda$-pure-epimorphic images'', then $(\dagger)$ turns out to be necessary for the deconstructibility of $\mathcal F$ in the absence of $0^\sharp$ by \cite[Corollary~3.5]{BS23}.

\smallskip

Another corollary of Theorem~\ref{t:closureprop} constitutes a~generalization of \cite[Theorem~1.6]{Cox22}.

\begin{corollary}\label{c:Cox} Assume that there is a proper class of strongly compact cardinals. Let $\mathcal S$ be a set of modules. Then there exists a regular cardinal $\lambda$ such that ${}^\perp\mathcal S$ is closed under $\lambda$-pure-epimorphic images and $\lambda$-pure submodules.
\end{corollary}

\begin{proof} By our assumption, we can find a strongly compact cardinal $\lambda$ such that each module in $\mathcal S$ has cardinality $<\lambda$. Consequently, by \cite[Proposition~2.1]{Saroch20}, each module in $\mathcal S$ is $\lambda$-pure-injective. The rest follows from Theorem~\ref{t:closureprop}$(1)$.
\end{proof}

\section{Completeness of the Gorenstein projective cotorsion pair and $\lambda$-pure-injective modules}
\label{sec:5}

In this section we deal with the completeness of the Gorenstein projective cotorsion pair. We prove that if every projective module is $\lambda$-pure-injective for some infinite regular cardinal $\lambda$, then $\mathfrak{GP} = (\GProj,\GProj^\perp)$ is complete. Moreover, if we assume that there exist many compact cardinals, then we show that all projective modules are $\lambda$-pure-injective for some $\lambda$ and, consequently, $\mathfrak{GP}$ is complete. In particular, one cannot disprove in ZFC that $\mathfrak{GP}$ is complete unless it turns out that the existence of these large cardinals contradicts the axioms of ZFC, which is extremely unlikely.

The following lemma serves as an important building block in our main result on the completeness of the Gorenstein projective cotorsion pair.

\begin{lemma} \label{l:aux}
Assume that $\lambda$ is an infinite regular cardinal and $\mu$ is an infinite cardinal such that $|R|\leq\mu = \mu^{<\lambda}$. Let $\mathcal D\colon 0\to N\overset{\subseteq}{\to} \bigoplus_{\delta\in\Delta}P_\delta\overset{\pi}{\to} N\to 0$ be a~short exact sequence of $R$-modules where $P_\delta$ is $\mu$-generated for each $\delta\in\Delta$. Let $X$ be a~subset of $N$ of cardinality $\leq \mu$.

Then there exists an exact subobject $\mathcal F\colon 0\to A\overset{\subseteq}{\to} \bigoplus_{\delta\in E}P_\delta \to A\to 0$ of $\mathcal D$ such that $|E|\leq\mu$, $X\subseteq A$ and $A$ is $\lambda$-pure in $N$.
\end{lemma}

\begin{proof}
We construct $\mathcal F$ as the union of an increasing chain \[(\mathcal F_\alpha\colon 0\to A_\alpha\overset{\subseteq}{\to}\bigoplus_{\delta\in E_\alpha}P_\delta\to A_\alpha\to 0\mid 0<\alpha<\lambda)\] of exact subobjects of $\mathcal D$. We start by letting $A_0$ be the submodule of $N$ generated by $X$ and $E_0=\varnothing$. (Notice that we do not define any $\mathcal F_0$ though.)

Let $0<\alpha<\lambda$ and assume that $A_\beta$ and $E_\beta$ are defined for all $\beta<\alpha$ in such a~way that $|E_\beta|\leq\mu$ holds; then $|A_\beta|\leq\mu$ holds as well. Suppose first that $\alpha = \gamma+1$ for some ordinal $\gamma$, i.e.\ $\alpha$ is non-limit. Using Lemma \ref{l:purification} we can find find $A_\gamma^\prime\supseteq A_\gamma$ of cardinality $\leq\mu$ such that $A_\gamma^\prime$  is $\lambda$-pure in $N$.

Next, we find $D_0\subseteq\Delta$ containing $E_\gamma$, of cardinality $\leq\mu$ and such that $A_\gamma^\prime\subseteq\Img \pi\restriction\bigoplus_{\delta\in D_0}P_\delta =:A_\gamma^0$. Further, we find $D_1\subseteq\Delta$ containing $D_0$, of cardinality $\leq\mu$ and such that $A_\gamma^0 \subseteq \bigoplus_{\delta\in D_1}P_\delta$ and $A_\gamma^1 := \Img \pi\restriction\bigoplus_{\delta\in D_1}P_\delta\supseteq N\cap \bigoplus_{\delta\in D_0}P_\delta$. We continue by finding $D_2$ containing $D_1$, of cardinality $\leq\mu$ and such that $A_\gamma^1\subseteq \bigoplus_{\delta\in D_2}P_\delta$ and $A_\gamma^2:= \Img \pi\restriction\bigoplus_{\delta\in D_2}P_\delta\supseteq N\cap \bigoplus_{\delta\in D_1}P_\delta$, and so on. Finally, we put $E_\alpha = \bigcup_{n<\omega}D_n$ and $A_\alpha = \bigcup_{n<\omega}A_\gamma^n$, thus obtaining a short exact sequence $\mathcal F_\alpha\colon 0\to A_\alpha\overset{\subseteq}{\to} \bigoplus_{\delta\in E_\alpha}P_\delta\to A_\alpha\to 0$ which is a~subobject of $\mathcal D$.

For limit $\alpha<\lambda$, we define $\mathcal F_\alpha$ as the directed union of all the $\mathcal F_\beta$ where $\beta<\alpha$. Finally, we put $\mathcal F = \bigcup_{\alpha<\lambda}\mathcal F_\alpha$.

\smallskip

It is clear that $\mathcal F$ has the form $0\to A\overset{\subseteq}{\to}\bigoplus_{\delta\in E}P_\delta\to A\to 0$ where $|E|\leq\mu$. It remains to check that $A$ is $\lambda$-pure in $N$. So let $\mathcal S$ be a~system of cardinality $<\lambda$ consisting of $R$-linear equations with parameters in $A$ and assume that $\mathcal S$ has a~solution in $N$. Since $\lambda$ is a~regular cardinal, there exists $\alpha<\lambda$ such that $A_\alpha$ contains all the parameters from $\mathcal S$. By the construction, $\mathcal S$ has a solution in $A_{\alpha+1}\subseteq A$. It follows that $A$ is $\lambda$-pure in $N$.
\end{proof}

We can now prove the main result of this section:

\begin{theorem}\label{t:complete}
Assume that $\lambda$ is a regular infinite cardinal and $\mathcal C$ is a~class consisting of $\lambda$-pure-injective modules such that $\Proj\subseteq\mathcal C$. Then $\GProj_{\mathcal C}$ is a~deconstructible class of modules.\footnote{Sean Cox has recently proved this in \cite{Cox21} for an arbitrary class $\mathcal C$ assuming Vop\v enka's principle.} In particular, $\mathfrak H = (\GProj_{\mathcal C}, \GProj_{\mathcal C}^\perp)$ is a complete hereditary cotorsion pair with $\mathcal H^\perp$ being a resolving class.

In particular, $\GProj$ is deconstructible and $(\GProj, \GProj^\perp)$ is a complete cotorsion pair provided that projective modules are $\lambda$-pure-injective.
\end{theorem}

\begin{proof} In the view of Lemma~\ref{l:PropertiesCGorenstein} and Theorem~\ref{t:CotPair}, we are left to show that $\GProj_{\mathcal C}$ is deconstructible and $\mathfrak H$ is complete.

Let $\mu$ be an arbitrary infinite cardinal such that $|R|\leq \mu = \mu^{<\lambda}$. We will show that $\mathfrak{H}$ is generated by a representative set of $\mu$-presented (i.e.\ of cardinality $\leq\mu$) strongly $\mathcal C$-Gorenstein projective modules. Then the result will follow from \cite[Theorem~6.11 and Theorem~7.13]{GobelTrlifaj21}. Since each module in $\GProj_{\mathcal C}$ is a~direct summand in a strongly $\mathcal C$-Gorenstein projective module, it is enough to show that each strongly $\mathcal C$-Gorenstein projective module $M$ is filtered by $\mu$-presented strongly $\mathcal C$-Gorenstein projective modules.

For the sake of non triviality, assume that $\kappa = |M|>\mu$. Fix a short exact sequence $\mathcal E\colon 0\to M\overset{\subseteq}{\to} Q\to M\to 0$ with $Q$ projective, and a decomposition $Q = \bigoplus_{\gamma\in\Gamma}P_\gamma$ where each $P_\gamma$ is countably generated. Let us also fix a~generating set $\{m_\alpha\mid \alpha<\kappa\}$ of $M$. We recursively build a filtration $(\mathcal E_\alpha\colon 0\to M_\alpha\to Q_\alpha\to M_\alpha\to 0\mid \alpha\leq\kappa)$ of $\mathcal E$ with the following properties for each $\alpha<\kappa$:
\begin{enumerate}
	\item[(i)] $\mathcal E_\alpha$ is short exact and $Q_\alpha = \bigoplus_{\gamma\in\Gamma_\alpha}P_\gamma$ for some $\Gamma_\alpha\subseteq\Gamma$;
	\item[(ii)] $M_{\alpha+1}/M_\alpha$ is a $\mu$-presented strongly $\mathcal C$-Gorenstein projective module;
	\item[(iii)] $M/M_{\alpha}$ is a strongly $\mathcal C$-Gorenstein projective module;
  \item[(iv)] $m_\alpha\in M_{\alpha+1}$.
\end{enumerate}

We start with $\mathcal E_0$ consisting of trivial modules. Assume that $\mathcal E_\beta$ is already constructed for all $\beta<\alpha$. If $\alpha$ is a limit ordinal, we simply put $\mathcal E_\alpha = \bigcup_{\beta<\alpha}\mathcal E_\beta$. The module $M_\alpha$ is strongly $\mathcal C$-Gorenstein projective by (ii) and Eklof Lemma. As a~result, each homomorphism $h\colon M_\alpha \to C$, where $C\in\mathcal C$, extends to $Q_\alpha$ and subsequently to $Q$, in particular to $M$. This shows that $M/M_\alpha\in {}^\perp\mathcal C$, whence (iii), and so (i)--(iv), is satisfied for~$\alpha$.

Now let $\alpha = \delta + 1$ for some $\delta$. Consider the short exact sequence $\mathcal D\colon 0\to M/M_\delta\to \bigoplus_{\gamma\in\Gamma\setminus\Gamma_\delta}P_\gamma\to M/M_\delta \to 0$ and put $N = M/M_\delta$. Using Lemma~\ref{l:aux}, we find an exact subobject $\mathcal F\colon 0\to A\overset{\subseteq}{\to} \bigoplus_{\gamma\in E}P_\gamma \to A\to 0$ of $\mathcal D$ such that $|E|\leq\mu$, $m_\delta+M_\delta\in A$ and $A$ is $\lambda$-pure in~$N$.

Since all modules in $\mathcal C$ are $\lambda$-pure-injective, $N/A\in {}^\perp\mathcal C$ by Theorem \ref{t:closureprop}, and thus it is strongly $\mathcal C$-Gorenstein projective. It follows that $A$ is strongly $\mathcal C$-Gorenstein projective as well since $\GProj_{\mathcal C}$ is closed under kernels of epimorphisms.

We define the short exact sequence $\mathcal E_\alpha$ as the preimage of $\mathcal F$ in the canonical projection from $\mathcal E$ onto $\mathcal D$. The conditions (i)--(iv) then immediately follow: indeed, we have $\Gamma_\alpha = \Gamma_\delta\cup E$, $M/M_\alpha\cong N/A$, and $M_{\alpha}/M_{\delta}\cong A$ is a~$\mu$-presented module since $|E|\leq\mu$.
\end{proof}

At this point, it is worth noting that the proof above relies on the special form of strongly $\mathcal C$-Gorenstein projective modules to overcome the main technical obstacle which arises when working with $\lambda$-pure submodules for $\lambda$ uncountable: the fact that the union of a short (i.e.\ of length $<\lambda$) ascending chain of $\lambda$-pure submodules need not be $\lambda$-pure. Notice that we do not claim that the submodules $M_\alpha$ are $\lambda$-pure in $M$ for limit ordinals $\alpha$. Nevertheless, we are able to derive the crucial property that $M/M_\alpha$ is strongly $\mathcal C$-Gorenstein projective.

\smallskip

We are going to show that, using an additional set-theoretical assumption, we can assure that there exists an infinite regular cardinal $\lambda$ such that all projective modules are $\lambda$-pure-injective.

Given an uncountable regular cardinal $\mu$, a cardinal $\lambda$ is called \emph{$\mathcal L_{\mu\omega}$-compact} if every $\lambda$-complete filter (on any set) can be extended to a~$\mu$-complete ultrafilter. Note that necessarily $\lambda\geq\mu$ and that a cardinal greater than an $\mathcal{L}_{\mu\omega}$-compact cardinal is again $\mathcal L_{\mu\omega}$-compact.

\begin{proposition} \label{p:largecard}
Assume that, for each cardinal $\kappa$, there is an uncountable regular cardinal $\mu > \kappa$ for which there exists an $\mathcal L_{\mu\omega}$-compact cardinal. Then, for any ring, there exists an infinite regular cardinal~$\lambda$ such that each projective module is $\lambda$-pure-injective.
\end{proposition}

\begin{proof}
By \cite[Theorem~3.4]{Saroch15}, there exists a free $R$-module $F$ such that all projective $R$-modules belong to $\Prod(F)$. Consider a regular $\mathcal L_{\mu\omega}$-compact cardinal $\lambda$ with $\mu>|F|$. It follows from \cite[Proposition~2.1]{Saroch20} that a system of $R$-linear equations with parameters in $F$ has a solution in $F$ provided that each its subsystem of cardinality $<\lambda$ has a solution in $F$, i.e.\ $F$ is a $\lambda$-pure-injective module.\footnote{In fact, each module $M$ with $|M|<\mu$ is $\lambda$-pure-injective in this case by the cited proposition; in analogy with the well-known result that finite (not just finitely generated, but finite!) modules are pure-injective.} Consequently all modules in $\Prod(F)$, projective ones in particular, are $\lambda$-pure-injective as well.
\end{proof}

\begin{remark}\label{r:largecard} The large cardinal assumption of the proposition above is satisfied, if there exists a proper class of strongly compact cardinals, i.e.\ regular uncountable cardinals $\kappa$ which are $\mathcal L_{\kappa\omega}$-compact. This, in turn, follows from the famous (and much stronger) Vop\v{e}nka's Principle.
\end{remark}

We are ready to provide the promised consistency result concerning the completeness of Gorenstein projective cotorsion pair over any ring. Unfortunately, we were unable to get rid of the additional large-cardinal assumption. It might be the case that the result is independent of ZFC, similarly as what Whitehead problem turned out to be. This Gorenstein setting seems to be a lot more intricate though.

\begin{corollary}\label{c:complete}
Assume that, for each cardinal $\kappa$, there is an uncountable regular cardinal $\mu > \kappa$ for which there exists an $\mathcal L_{\mu\omega}$-compact cardinal. Then, for any ring, the class $\GProj$ is deconstructible and thus the cotorsion pair $(\GProj,\GProj^\perp)$ is complete.
\end{corollary}

\begin{proof}
Follows immediately from Proposition~\ref{p:largecard} and Theorem~\ref{t:complete}.
\end{proof}

If we are interested only in $\ModR$ for a~particular ring $R$, we can also use the following proposition.

\begin{proposition} \label{p:onestrongcomp} Let $R$ be a ring and $\lambda$ be a~strongly compact cardinal such that $|R|<\lambda$. Then all projective $R$-modules are $\lambda$-pure-injective.
\end{proposition}

\begin{proof} It is enough to prove the result for free modules. By \cite[Proposition~2.1]{Saroch20}, all modules of cardinality strictly less than $\lambda$ are $\lambda$-pure-injective. Given any free module $M$, it is a~$\lambda$-directed union of a~system of free submodules of cardinality $<\lambda$: indeed, we can assume that $M = R^{(I)}$ and consider the system $(R^{(J)} \mid J \in \mathcal P_\lambda(I))$ of submodules of $M$ where we put $\mathcal P_\lambda(I) = \{J \subseteq I \mid |J|<\lambda\}$. In particular, $M$ embeds $\lambda$-purely into the ultraproduct $U = \prod_{J\in\mathcal P_\lambda(I)} R^{(J)}/\mathcal U$ where $\mathcal U$ is a~$\lambda$-complete ultrafilter in $\mathcal P_\lambda(I)$ extending the $\lambda$-complete filter in $\mathcal P_\lambda(I)$ with the basis \[\mathcal B = \{\uparrow\! J \mid J\in \mathcal P_\lambda(I)\}\mbox{ where }\uparrow\! J = \{K\in \mathcal P_\lambda(I)\mid J\subseteq K\};\] here, we use that $\lambda$ is strongly compact and (the proof of) \cite[Theorem~3.3.2]{Prest}\footnote{The fact that our directed system is even $\lambda$-directed translates into the $\lambda$-completeness of the (ultra)filters considered therein, and also into the $\lambda$-purity of the resulting embedding of $M$.}.

By \cite[Theorem~II.3.8]{EM}, $U$ is a~free module. In particular, the canonical epimorphism $\prod_{J\in\mathcal P_\lambda(I)} R^{(J)} \to U$ splits. Since each (free) module $R^{(J)}$ has cardinality strictly less than $\lambda$, it is $\lambda$-pure-injective; consequently, the product $\prod_{J\in\mathcal P_\lambda(I)} R^{(J)}$ and $U$ are $\lambda$-pure-injective as well.

Finally, the $\lambda$-pure embedding of $M$ into $U$ splits by \cite[Theorem~3.3]{SarochTrlifaj19}, and so $M$ is $\lambda$-pure-injective, too.
\end{proof}

\medskip

\section{Existence of enough $\lambda$-pure-injective modules}
\label{s:lambda}

In the last section of this paper, we investigate the behavior of $\lambda$-pure-injective modules under a set-theoretic assumption which is inconsistent with the one from Proposition~\ref{p:largecard}, namely under the assumption that \[0^{\sharp}\mbox{ does not exist.}\eqno{(*)}\]

The assumption $(*)$ is consistent with ZFC. It follows from the axiom of constructibility, V = L. In fact, $(*)$ is its weaker form, sort of: an equivalent restatement of $(*)$ says that the classes V and L are, in a precise sense, close to each other but not necessarily equal. Consequently, the assumption $(*)$ does not imply the generalized continuum hypothesis (GCH) as V = L does, it yields only its weaker form, the singular cardinal hypothesis (SCH), which we capitalize on in what follows. For more information on $(*)$, we refer to \cite[\S 18]{Jech} or \cite[VI.3.16]{EM}.

The assumption $(*)$ was recently used in (the proof of) \cite[Proposition~1.5]{SarochTrlifaj19} to show that there does not exist a regular uncountable cardinal $\lambda$ such that the right $R$-module $R^{(\omega)}$ embeds $\lambda$-purely into a $\lambda$-pure-injective module, provided that $R$ is not right perfect. In particular, the free module $R^{(\omega)}$ is not $\lambda$-pure-injective for any regular uncountable cardinal $\lambda$ over such a ring, contradicting the conclusion of Proposition~\ref{p:largecard}.

Before stating our last main result, which generalizes \cite[Proposition~1.5]{SarochTrlifaj19}, let us recall that the category $\Modr R$ is said to have \textit{enough $\lambda$-pure-injective objects} if each $M\in\Modr R$ can be $\lambda$-purely embedded into a $\lambda$-pure-injective module. Also note that Theorem~\ref{t:0sharp} below is inconsistent with the main result from \cite{Hos19} which is flawed, mainly because of the nonsensical Lemma~2.8 therein.

Moreover, let us recall some set-theoretical notions which will be used in the proof of the next result. Let $\delta$ be a limit ordinal with uncountable cofinality. A~subset $C\subseteq\delta$ is called \textit{closed unbounded} if $\sup C=\delta$ and $\sup Y \in C$ holds for each $Y \subseteq C$ such that $\sup Y \in \delta$. A subset $E \subseteq \delta$ is \textit{stationary} in $\delta$ if $E \cap C \neq \varnothing$ for each closed unbounded subset $C$ of $\delta$. Finally, a subset $E \subseteq \delta$ is called \textit{non-reflecting} if $E \cap \gamma$ is not stationary in $\gamma$ whenever $\gamma<\delta$ is a limit ordinal of uncountable cofinality. For instance, if $\kappa$ is an infinite regular cardinal, then the set $E(\kappa) = \{\alpha<\kappa^+\mid \cf(\alpha) = \kappa\}$ is stationary and non-reflecting in $\kappa^+$. On the other hand, if $\kappa$ is uncountable, one needs an additional set-theoretic assumption to get a~non-reflecting stationary subset of $\kappa^+$ consisting of limit ordinals with (some prescribed infinite) cofinality $<\kappa$. The non-existence of $0^\sharp$ can be one such assumption.

\begin{proposition} \label{p:0sharp} \emph{(}$0^{\sharp}$ does not exist\emph{)} Let $R$ be a ring and $\lambda$ be an infinite regular cardinal. Then each $\lambda$-pure-injective module is pure-injective.
\end{proposition}

\begin{proof} Let $B$ be a $\lambda$-pure-injective module. For the sake of nontriviality, we can assume that $\lambda$ is uncountable. Put \[\mathcal A = \{A\in \Modr R \mid \PExt^1_R(A,B) = 0\}.\] To show that $B$ is pure-injective, we verify that $\mathcal A = \Modr R$.

From the properties of the bifunctor $\PExt^1_R$ and the $\lambda$-pure-injectivity of $B$, it follows that
\begin{enumerate}
	\item $\mathcal A$ contains all finitely presented modules (since these are pure-projective),
	\item $\mathcal A$ is closed under arbitrary direct sums and direct summands,
	\item $\mathcal A$ is closed under pure filtrations, i.e.\ if $(A_\alpha\mid\alpha\leq\sigma)$ is a~pure $\mathcal A$-filtration, then $A_\sigma\in\mathcal A$,
	\item $\lambda$-pure-epimorphic images of modules from $\mathcal A$ are again in $\mathcal A$.
\end{enumerate}
Note that $(3)$ follows from $\mathcal A$ being closed under pure extensions by the proof of the usual Eklof Lemma. We prove that $\mathcal A = \Modr R$ by checking that $\mathcal A$ is closed under arbitrary direct limits. It is well-known that to show this, it is enough to check that $\mathcal A$ is closed under direct limits of continuous well-ordered directed systems of the form $\mathcal N = (N_\alpha, g_{\alpha\beta}\colon N_\beta \to N_\alpha\mid \beta<\alpha<\theta)$ where $\theta$ is a~regular infinite cardinal. So let such a~directed system $\mathcal N$ of modules from $\mathcal A$ be given.

Put $N = \bigoplus_{\alpha<\theta} N_\alpha$ and $C = \varinjlim \mathcal N$. In particular, $N\in\mathcal A$ by $(2)$ above. Moreover, the direct limit $C$ is a~$\theta$-pure-epimorphic image of $N$, whence $C\in\mathcal A$ holds as well \emph{whenever} $\theta\geq \lambda$ (by the part $(4)$ above). In the rest of the proof, we can thus assume that $\theta<\lambda$ and, for the sake of nontriviality, that $C\neq 0$.

Let $\kappa\geq |B|+|N|+|R|$ be a~strong limit singular cardinal with $\cf(\kappa)\geq \lambda$. It follows that $\kappa = \kappa^{<\theta} = \kappa^{<\lambda}$; in particular, $\lambda<\kappa$. Put $\mu = \kappa^+$. The non-existence of $0^\sharp$ gives us a~non-reflecting stationary subset $E\subset \mu$ consisting of ordinals with cofinality~$\theta$.

We use \cite[Proposition~3.3]{BS23}\footnote{Our $\mu$ plays the role of the $\lambda$ therein.} to obtain a~stationary subset $F\subseteq \mu$ and a~module $L$ possessing a~filtration $\mathcal L = (L_\alpha\mid \alpha\leq\mu)$ satisfying
\begin{enumerate}
	\item[(a)] $L_\alpha\in\mathcal A$ and $|L_\alpha|\leq\kappa$ for each $\alpha<\mu$, and
	\item[(b)] for every $\delta\in F$ and $\delta<\varepsilon<\mu$, $L_{\delta+1}/L_\delta$ is a direct summand of $L_{\varepsilon}/L_\delta$ isomorphic to $C$.
\end{enumerate}

Notice that $L = L_\mu\in\mathcal A$: indeed, $L$ is a~$\mu$-pure-epimorphic image (in particular, a~$\lambda$-pure-epimorphic image) of $\bigoplus_{\alpha<\mu}L_\alpha\in\mathcal A$ (capitalizing on $(2)$ and $(4)$ above).

By Lemma~\ref{l:purefilt}, we obtain a~pure filtration $(F_\alpha\mid \alpha\leq\mu)$ of $L$ such that $F_\alpha$ is $\lambda$-pure in $L$ whenever $\alpha$ is non-limit or $\cf(\alpha)\geq \lambda$; furthermore, $|F_\alpha|\leq\kappa$ for each $\alpha<\mu$. It follows that the set $\{\alpha<\mu\mid F_\alpha = L_\alpha\}$ is closed unbounded. Consequently, the set $S = \{\alpha<\mu \mid F_\alpha = L_\alpha, \cf(\alpha)\geq\lambda\}$ is unbounded (even stationary) in $\mu$. Fix the non-decreasing bijection $b\colon \mu\to S$ and define a~new filtration $\mathcal F = (M_\alpha\mid \alpha\leq\mu)$ by putting $M_0 = 0$, $M_{\alpha+1} = L_{b(\alpha)} = F_{b(\alpha)}$ for each $\alpha<\mu$, and taking unions in the limit steps.

The resulting filtration $\mathcal F$ is a~pure filtration of $L$ consisting of modules from~$\mathcal A$. Moreover, since each $M_\alpha$ where $\alpha$ is non-limit or $\cf(\alpha)\geq \lambda$ is $\lambda$-pure in $L$ and hence in $M_{\alpha+1}$, we get that $M_{\alpha+1}/M_\alpha\in\mathcal A$ by the part $(4)$ above. Otherwise put, each homomorphism from $M_\alpha$ into $B$ extends to $M_{\alpha+1}$, equivalently to $L$.

\smallskip

We claim that the set \[Z = \{\alpha<\mu \mid M_{\alpha+1}/M_\alpha\notin\mathcal A\}\] is not stationary. Assume, for the sake of contradiction, that it is stationary. By the paragraph above, we know that each $\alpha\in Z$ is a~limit ordinal and has cofinality strictly less than $\lambda$.

The non-existence of $0^\sharp$ implies SCH (cf.\ \cite[Corollary~18.33]{Jech}) which gives us $2^\kappa = \mu$ and, consequently, $\diamondsuit_\mu(Z)$ by \cite{Shelah10}; in particular, we have the weak diamond $\Phi_\mu(Z)$ which we are about to use in the sequel. 

Let $\mathcal E \colon 0 \to \Ker(\pi) \overset{\subseteq}{\to} \bigoplus_{\alpha<\mu}M_\alpha \overset{\pi}{\to} L \to 0$ be the canonical ($\mu$-pure) presentation of $L = M_\mu$ as the directed union of the filtration $\mathcal F$. Also, for each $\alpha<\mu$, put $\mathcal E_\alpha\colon 0 \to \Ker(\pi_\alpha) \overset{\subseteq}{\to} \bigoplus_{\beta<\alpha}M_\beta \overset{\pi_\alpha}{\to} M_{\alpha^\prime} \to 0$ where $\pi_\alpha = \pi\restriction \bigoplus_{\beta<\alpha}M_\beta$ and $\alpha^\prime$ denotes $\alpha$ if $\alpha$ is limit (or zero), and $\alpha^\prime = \alpha - 1$ otherwise. (See also \cite[Lemma~2.7]{BS23}.)

Since $M_\alpha\in\mathcal A$ for each $\alpha<\mu$, and each short exact sequence $\mathcal E_\alpha$ is pure, we see that each homomorphism from $\Ker(\pi_\alpha)$ into $B$ extends to $\bigoplus_{\beta<\alpha}M_\beta$. This allows us to use Lemma~\ref{l:wdiamond}, where our $Z$ plays the role of $E$ in the lemma, to deduce that the contravariant functor $\Hom_R(-,B)$ does not preserve the exactness of $\mathcal E$, in contradiction with $L\in\mathcal A$ (or, in fact, with $\mathcal E$ being $\lambda$-pure and $B$ being $\lambda$-pure-injective).

\smallskip

We showed that $Z$ is not stationary, whence we can choose a~subfiltration of $\mathcal F$, and in turn a~subfiltration $\mathcal K$ of $\mathcal L$, with the consecutive factors in $\mathcal A$. Using the part (b) above and the stationarity of the set $F$ provided to us by \cite[Proposition~3.3]{BS23}, we immediately get that $C = \varinjlim_{\alpha<\theta} N_\alpha$ is isomorphic to a~direct summand in a~consecutive factor of the filtration $\mathcal K$. In particular, $C\in\mathcal A$ and the proof is complete.
\end{proof}

\begin{remark}
In \cite[Theorem 3]{Megibben82}, the author proves, assuming the axiom of constructibility V = L, that every $\lambda$-pure-injective abelian group is pure-injective for any infinite regular cardinal $\lambda$. The preceding result extends Megibben's theorem to any ring assuming the weaker axiom $(*)$.
\end{remark}

Our main theorem now easily follows.

\begin{theorem}\label{t:0sharp} \emph{(}$0^{\sharp}$ does not exist\emph{)} The following conditions on a ring $R$ are equivalent:

\begin{enumerate}
	\item $R$ is right pure-semisimple.
	\item $\Modr R$ has enough $\lambda$-pure-injective objects for each regular cardinal~$\lambda\geq\aleph_0$.
	\item There exists an \textbf{uncountable} regular cardinal $\lambda$ such that $\Modr R$ has enough $\lambda$-pure-injective objects.
  \item Each countably presented pure-projective right $R$-module can be $\aleph_1$-purely embedded into a $\lambda$-pure-injective module for a suitable infinite regular cardinal $\lambda$.
\end{enumerate}
\end{theorem}

\begin{proof} Over a right pure-semisimple ring, all right modules are pure-injective, in particular $\lambda$-pure-injective for any $\lambda\geq\aleph_0$. Thus $(1)\Longrightarrow (2)$ holds. The implications $(2)\Longrightarrow (3)$ and $(3)\Longrightarrow (4)$ are trivial. It remains to show that $(4)\Longrightarrow (1)$.

Assuming that $(4)$ holds true, we prove that $R$ is right pure-semisimple by showing that each countably presented module is pure-projective. Let $C$ be a countably presented module. We can form a~pure short exact sequence \[0\longrightarrow \bigoplus_{n<\omega} F_n \overset{\varepsilon}{\longrightarrow} \bigoplus_{n<\omega} F_n \longrightarrow C \longrightarrow 0\] where all the modules $F_n$ are finitely presented.

Put $F=\bigoplus_{n<\omega}F_n$. Using $(4)$, there exists an $\aleph_1$-pure embedding $f\colon F\to B$ with $B$ a $\lambda$-pure-injective module. Proposition~\ref{p:0sharp} implies that $B$ is, in fact, pure-injective. Consequently, $f$ factorizes through the pure embedding $\varepsilon$ showing that $\varepsilon$ is actually $\aleph_1$-pure, and the short exact sequence above thus splits since $C$ is countably presented. It follows that $C$ is pure-projective.
\end{proof}

Theorem~\ref{t:0sharp} does not hold without the additional assumption $(*)$. To be more precise, the implication $(4)\Longrightarrow (1)$ does not hold since under the presence of a~strongly compact cardinal $\lambda$, all modules with cardinality less than $\lambda$ are $\lambda$-pure-injective regardless of the ring, cf.\ \cite[Proposition~2.1]{Saroch20}. In particular, the countably presented ones are such provided that $|R|<\lambda$. It is an open question whether the equivalence of $(1)$--$(3)$ can be proved in ZFC.

\smallskip

We end this section with another application of Theorem \ref{t:closureprop} when there are enough $\lambda$-pure-injective modules. It is an open problem whether it is consistent with ZFC that, over an arbitrary ring~$R$ (which is not right perfect), there exists a \emph{projective-testing module}, i.e.\ a~module $C\in\ModR$ such that ${}^\perp C = \Proj$. As a corollary of Theorem \ref{t:closureprop}, we show that this might be the case in some models of ZFC. Note that we do not know whether $(2)$ below is consistent with ZFC for uncountable $\lambda$ and $R$ which is not right pure semisimple. (For the notion of a~strongly compact cardinal, see Remark~\ref{r:largecard}.)

\begin{corollary}\label{c:testproj} Assume that $R$ is a ring and $\lambda$ is a regular cardinal such that:
\begin{enumerate}
	\item there exists a strongly compact cardinal $\kappa$ with $|R|<\kappa\leq\lambda$;
	\item $\Modr R$ has enough $\lambda$-pure-injective objects.
\end{enumerate}
Then there exists a projective-testing module in $\Modr R$.
\end{corollary}

\begin{proof} Let $\mathcal C$ denote the class of all $\lambda$-pure-injective modules and put $\mathcal A = {}^\perp\mathcal C$. By Theorem~\ref{t:closureprop}, we get $C\in\mathcal C$ such that $\mathcal A = {}^\perp C$. We claim that $\mathcal A = \Proj$.

Let $M\in\mathcal A$ be arbitrary and $0 \to K \overset{\subseteq}{\to} P \overset{\pi}{\to} M\to 0$ be a short exact sequence with $P$ projective. By $(2)$, there exists a $\lambda$-pure embedding $e\colon K\to N$ where $N\in\mathcal C$. Since $M\in\mathcal A$, the mapping $e$ can be extended to $P$ showing that $K$ is a $\lambda$-pure submodule of $P$. It follows from $(1)$ and \cite[Theorem~3.3]{SarochTrlifaj19} that $\pi$ splits whence $M$ is projective.
\end{proof}

\medskip

\appendix

\section{A general lemma involving the weak diamond principle}
\label{sec:append}

This appendix is devoted to the proof of the following lemma generalizing \cite[Lemma~2.9]{SarochTrlifaj19}. First, however, we have to recall a couple of definitions. Given a~set $X$ and a~cardinal $\mu$, we say that a~system $(X_\alpha\mid\alpha<\mu)$ is a \emph{$\mu$-filtration of $X$} if $X = \bigcup_{\alpha<\mu}X_\alpha$, $|X_\alpha|<\mu$, $X_\alpha\subseteq X_{\alpha+1}$ for each $\alpha<\mu$, and $X_\alpha = \bigcup_{\beta<\alpha}X_\beta$ whenever $\alpha<\mu$ is a~limit ordinal.

For a regular uncountable cardinal $\mu$ and a stationary subset $E \subseteq \mu$, we denote by $\Phi_\mu(E)$ the following statement: let $X$ be a set with $|X|\leq\mu$ and $(X_\alpha\mid \alpha<\mu)$ be a $\mu$-filtration of $X$; moreover, for each $\alpha\in E$ let a mapping $P_\alpha\colon \mathcal P(X_\alpha)\to \{0,1\}$ be given. Then there exists $\varphi\colon E\to \{0,1\}$ such that, for each $Y\subseteq X$, the set $\{\alpha\in E\mid P_\alpha(Y\cap X_\alpha) = \varphi(\alpha)\}$ is stationary in $\mu$.

The \emph{Weak Diamond Principle} $\Phi$ asserts that ``$\Phi_\mu(E)$ holds for each regular uncountable cardinal $\mu$ and stationary subset $E\subseteq\mu$''. It is consistent with ZFC since it follows, for instance, from the axiom of constructibility V = L. It is however much weaker than V = L. Also each of its instances, $\Phi_\mu(E)$, is considerably weaker than (and thus follows from) the corresponding instance, $\diamondsuit_\mu(E)$, of the famous Jensen's diamond principle.

\begin{lemma}\label{l:wdiamond} Let $B$ be an $R$-module with $|B|\leq \mu$ where $\mu$ is a regular uncountable cardinal. Let \[\mathfrak E = (\mathcal E_\alpha\colon 0\to K_\alpha\overset{\ep_\alpha}{\to} H_\alpha \overset{\pi_\alpha}{\to} M_\alpha\to 0, (u_{\beta\alpha},v_{\beta\alpha},w_{\beta\alpha})\colon \mathcal E_\alpha\to \mathcal E_\beta \mid \alpha<\beta<\mu)\] be a continuous well-ordered directed system consisting of short exact sequences of modules where $H_\alpha$ is $<\mu$-generated for every $\alpha<\mu$. Moreover, assume that, for each $\alpha<\mu$, $u_{\alpha+1,\alpha}$ and $v_{\alpha+1,\alpha}$ are inclusions and the maps $\Hom_R(v_{\alpha+1,\alpha},B)$ and $\Hom_R(\ep_\alpha,B)$ are surjective. Let $\mathcal E\colon 0\to K\overset{\ep}{\to} H \overset{\pi}{\to} M\to 0 = \varinjlim \mathfrak E$.

If a set $E\subseteq \{\alpha<\mu \mid \Hom_R(w_{\alpha+1,\alpha},B)\mbox{ is not surjective}\,\}$ is stationary in~$\mu$ and $\Phi_\mu(E)$ holds, then $\Hom_R(\ep, B)$ is not surjective.
\end{lemma}

\begin{proof} The proof is a rather small modification of the one of \cite[Lemma~2.9]{SarochTrlifaj19}. We present the full argument here for reader's convenience.

Since $\Hom_R(\ep_\alpha, B)$ is surjective, we can fix, for each homomorphism $f\colon K_\alpha \to B$, an $f^e\in\Hom_R(H_\alpha, B)$ such that $f =f^e\ep_\alpha$. Furthermore, we fix, for each $\alpha\in E$, a~$k_\alpha\in\Hom_R(M_\alpha,B)$ which cannot be factorized through $w_{\alpha+1,\alpha}$.

Using the fact that $\mu$ is regular, we also fix a generating set $V$ of the module $H$ in such a way that $V_\alpha:= V\cap H_\alpha$ generates $H_\alpha$ and $|V_\alpha|<\mu$ for any $\alpha<\mu$. Then $|V|\leq\mu$. Further, we consider a $\mu$-filtration $(B_\alpha\mid\alpha<\mu)$ of the set $B$ (the $B_\alpha$ need not be submodules of $B$, just subsets) and put $X = V\times B$ and $X_\alpha = V_\alpha\times B_\alpha$ for each $\alpha<\mu$. Notice that $(X_\alpha\mid \alpha<\mu)$ is a $\mu$-filtration of $X$.

Let $\alpha\in E$ be arbitrary. We define the mapping $P_\alpha\colon \mathcal P(X_\alpha)\to \{0,1\}$ as follows: if $Z\subseteq X_\alpha$ and there is no homomorphism $z\colon H_\alpha\to B$ such that $Z = z\restriction V_\alpha$, we put $P_\alpha(Z) = 0$; otherwise, we fix the unique $z\in\Hom_R(H_\alpha, B)$ such that $Z = z\restriction V_\alpha$ and put $y = (z\ep_\alpha)^e$. Then $y-z$ is zero on $\Img(\ep_\alpha)$, and thus it defines a unique homomorphism from $M_\alpha$ to $B$. We put $P_\alpha(Z) = 1$ if and only if this homomorphism can be factorized through $w_{\alpha+1,\alpha}$.

Using $\Phi_\mu(E)$, we get $\varphi\colon E\to \{0,1\}$ for our choice of the mappings $P_\alpha$. We recursively construct a homomorphism $f\colon K\to B$ which cannot be factorized through $\ep$. We start with $f_0\colon K_0 \to B$ the zero map. If $\alpha\leq\mu$ is a limit ordinal, we put $f_\alpha = \bigcup_{\beta<\alpha}f_\beta$. Let us assume that $f_\alpha$ is already constructed and $\alpha<\mu$. We define $f_{\alpha+1}\colon  K_{\alpha+1}\to B$ as follows:

Let us put $f_\alpha^\prime = f_\alpha^e$ if $\alpha\not\in E$ or $\varphi(\alpha) = 0$; otherwise, we put $f_\alpha^\prime = f_\alpha^e+k_\alpha\pi_\alpha$. Using the surjectivity of $\Hom_R(v_{\alpha+1,\alpha},B)$, we extend $f_\alpha^\prime$ arbitrarily to a~homomorphism $f_\alpha^+\colon H_{\alpha+1}\to B$ and define $f_{\alpha+1}$ as $f_\alpha^+\ep_{\alpha+1}$.

Finally, we put $f = f_\mu\colon K\to B$. For the sake of contradiction, assume that there exists $g\in\Hom_R(H,B)$ such that $g\ep = f$. It is easy to see that the set $C = \{\alpha<\mu\mid g(V_\alpha)\subseteq B_\alpha\}$ is closed unbounded. Using the property of $\varphi$ for $Y = g\restriction V$, we obtain a $\delta\in C\cap E$ such that $P_\delta(g\restriction V_\delta) = \varphi(\delta)$. Obviously, $f_\delta^+-g\restriction H_{\delta+1}$ is zero on $\Img(\ep_{\delta+1})$. Thus, it defines a unique $h\in\Hom_R(M_{\delta+1},B)$. Then $k = hw_{\delta+1,\delta}$ where $k\colon M_\delta\to B$ is such that $k\pi_\delta = f_\delta^\prime-g\restriction H_\delta$.

If $\varphi(\delta) = 0$, then $k\pi_\delta = f_\delta^e-g\restriction H_\delta$ in the contradiction with $P_\delta(g\restriction V_\delta) = 0$ which has meant that $k$ cannot be factorized through $w_{\delta+1,\delta}$.

If on the other hand $\varphi(\delta) = 1$, then $k\pi_\delta = k_\delta\pi_\delta + f_\delta^e - g\restriction H_\delta$. Since $P_\delta(g\restriction V_\delta) = 1$, we know that $k-k_\delta$ can be factorized through $w_{\delta+1,\delta}$ which immediately implies that $k_\delta$ has this property as well, in contradiction with its choice.
\end{proof}

\medskip

\noindent\emph{Acknowledgement}: The authors would like to thank Ioannis Emmanouil for careful reading of the vast majority of the paper and providing a~valuable feedback.

\bibliographystyle{plain} \bibliography{references}
\end{document}

%% file: macros.tex
\newcommand{\ep}{\varepsilon}

\newcommand{\Hom}{\operatorname{Hom}}

\newcommand{\Ext}{\operatorname{Ext}}
\newcommand{\PExt}{\operatorname{PExt}}
\newcommand{\Tor}{\operatorname{Tor}}

\newcommand{\Ker}{\operatorname{Ker}}

\newcommand{\cf}{\operatorname{cf}}

\newcommand{\Img}{\operatorname{Im}}
\newcommand{\Coker}{\operatorname{Coker}}

\newcommand{\Modr}[1]{\mathrm{Mod}\textrm{-}{#1}}

\newcommand{\Flat}{\mathrm{Flat}}

\newcommand{\Inj}{\mathrm{Inj}}

\newcommand{\Proj}{\mathrm{Proj}}

\DeclareMathOperator{\GProj}{GProj}

\DeclareMathOperator{\PGFlat}{PGFlat}
\DeclareMathOperator{\AGProj}{AGProj}
\DeclareMathOperator{\DGProj}{DGProj}

\newcommand{\ModR}{\mathrm{Mod}\textrm{-}R}

\newcommand{\Filt}[1]{\operatorname{Filt}{#1}}

\newcommand{\Prod}{\mathrm{Prod}}

\DeclareMathOperator{\Z}{Z}

\theoremstyle{plain}
\newtheorem{theorem}{Theorem}[section]
\newtheorem{lemma}[theorem]{Lemma}
\newtheorem{proposition}[theorem]{Proposition}
\newtheorem{corollary}[theorem]{Corollary}

\newtheorem{definition}[theorem]{Definition}

\newtheorem{remark}[theorem]{Remark}